\documentclass[12pt]{amsart}
\usepackage{amssymb}

\newtheorem{theorem}{Theorem}[section]
\newtheorem{proposition}[theorem]{Proposition}

\newtheorem{dfn}{Definition}[section]

\newtheorem{rem}[dfn]{Remark}

\def\Z{\ensuremath{\mathbb Z}}
\def\C{\ensuremath{\mathbb C}}
\def\P{\ensuremath{\mathbb P}}
\def\R{\ensuremath{\mathbb R}}

\begin{document}
\baselineskip=15pt

\title[Polygons in Minkowski three space and rank two parabolic
bundles]{Polygons in Minkowski three space and parabolic
Higgs bundles of rank two on ${\mathbb C}{\mathbb P}^1$}

\author[I. Biswas]{Indranil Biswas}

\address{School of Mathematics, Tata Institute of Fundamental
Research, Homi Bhabha Road, Bombay 400005, India}

\email{indranil@math.tifr.res.in}

\author[C. Florentino]{Carlos Florentino}

\address{Departamento Matem\'atica, Centro de An\'alise Matem\'atica, Geometria e Sistemas din\^amicos -- LARSYS, Instituto Superior
T\'ecnico, Av. Rovisco Pais, 1049-001 Lisbon, Portugal}

\email{cfloren@math.ist.utl.pt}

\author[L. Godinho]{Leonor Godinho}

\address{Departamento Matem\'atica, Centro de An\'alise Matem\'atica, Geometria e Sistemas din\^amicos -- LARSYS, Instituto Superior
T\'ecnico, Av. Rovisco Pais, 1049-001 Lisbon, Portugal}

\email{lgodin@math.ist.utl.pt}

\author[A. Mandini]{Alessia Mandini}

\address{Departamento Matem\'atica, Centro de An\'alise Matem\'atica, Geometria e Sistemas din\^amicos -- LARSYS, Instituto Superior
T\'ecnico, Av. Rovisco Pais, 1049-001 Lisbon, Portugal}

\email{amandini@math.ist.utl.pt}


\thanks{Partially supported by FCT projects PTDC/MAT/108921/2008, PTDC/MAT/099275/2008, PTDC/MAT/120411/2010 and FCT grant SFRH/BPD/44041/2008}
\date{}

\begin{abstract}
Consider the moduli space of parabolic Higgs bundles $(E,\Phi)$ of rank two on $\C\P^1$ such that the underlying holomorphic vector bundle
for the parabolic vector bundle $E$ is trivial. It is equipped with the natural involution
defined by  $(E,\Phi)\longmapsto (E,-\Phi)$. We study the fixed point locus of this involution. In \cite{GM},  this moduli space with involution was identified with the moduli space of
hyperpolygons equipped with a certain natural involution. Here we identify the fixed point locus with the moduli spaces of polygons in Minkowski $3$-space. This identification
yields  information on the connected components of the fixed point locus.
\end{abstract}

\maketitle

\section{Introduction}

Parabolic vector bundles over a compact Riemann surface $\Sigma$ with $n$ 
marked points are holomorphic
vector bundles over $\Sigma$ with a weighted flag structure over each 
of the marked 
points. They were introduced by Seshadri, \cite{S}, and are of interest for 
many 
reasons. For instance, there is a natural bijective correspondence between 
the isomorphism classes of polystable parabolic bundles of parabolic degree 
zero and the equivalence classes of unitary representations of the fundamental 
group of the $n$-punctured surface.

Parabolic Higgs bundles are pairs of the form $(E,\Phi)$, where $E$ is a 
parabolic vector bundle on $\Sigma$
and $\Phi$ is a meromorphic $End(E)$-valued $1$-form 
holomorphic outside the $n$ marked points such that $\Phi$
has at most a simple pole with nilpotent residue (with respect to the flag)
at each of the marked points. 
There is a natural relationship between the polystable parabolic Higgs bundles 
of parabolic degree zero and the representations 
of the fundamental group of 
the $n$-punctured surface in the general linear groups \cite{Si}. Parabolic 
Higgs bundles have been studied in other works such as \cite{BY,N1,K1,GM}.

We will be particularly interested in the case of parabolic Higgs 
bundles of rank two over a $n$-pointed Riemann surface of genus zero.

Consider the split real form ${\rm PGL}(2,\R)$ of ${\rm PGL}(2,\C)$ defined by the
involution $A\, \longmapsto\, \overline{A}$. It produces 
the anti-holomorphic involution on the moduli space of representations
corresponding to the holomorphic involution
\begin{equation}\label{eq:0.1}
\sigma:(E\, ,\Phi) \,\longmapsto\, (E\, ,-\Phi)
\end{equation}
of the moduli space of parabolic Higgs bundles \cite{Hi1}.
Note that $\sigma$ is the restriction to $-1$ of the $\text{U}(1)= S^1$-action 
on the moduli space of parabolic Higgs bundles defined by
$$
\lambda\cdot (E, \Phi)\, =\, (E,\lambda \Phi), \quad \lambda \,\in\, 
S^1\, .
$$
The isomorphism classes of stable parabolic Higgs bundles fixed by this 
involution correspond to ${\rm SU}(2)$ or ${\rm SL}(2,\R)$ representations, the 
former corresponding to parabolic Higgs bundles with zero Higgs field;
see \cite{Hi1}.

We study the fixed points in the special case where the 
underlying vector bundle is holomorphically trivial.
Let $\mathcal{H}(\beta)$ be the moduli space of 
parabolic Higgs bundles $(E,\Phi)$, where $E$ is a
holomorphically trivial vector bundle over $\C 
\P^1$ of rank two with a weighted complete flag structure over each of
the $n$
marked points $x_1,\cdots,x_n$ 
\begin{align*}
E_{x_i,1}\,\supsetneq\, E_{x_i,2} \,\supsetneq\, 0\, , \\
0 \,\leq\, \beta_1(x_i) \,< \,\beta_2(x_i) \,<\,1\, .
\end{align*}
As shown in \cite{GM},
there is an isomorphism between $\mathcal{H}(\beta)$
and the hyperpolygon space $X(\alpha)$, with 
$\alpha_i\,=\,\beta_2(x_i)-\beta_1(x_i)$, defined as a hyper-K\"{a}hler 
quotient 
of $T^* \C^{2n}$ by 
$$K \,:= \,\Big({\rm U}(2) \times {\rm U}
(1)^n\Big)/{\rm U}(1)\,=\, \Big({\rm SU}(2) \times 
{\rm U}(1)^n\Big)/ (\Z/2\Z)\, ,$$
where $\Z/2\Z$ acts by multiplication of each factor by $-1$.
(See also Sections~\ref{hyperpolygons} and \ref{PHBs} for details.)

Using this correspondence between the two moduli spaces, we study in 
Section~\ref{sec:inv} the fixed point set of the corresponding involution
of $X(\alpha)$ defined by
\begin{equation}\label{eq:0.2}
\sigma\,:\,[p,q] \,\longmapsto\, [-p,q]\, ,
\end{equation}
with $(p,q)\in T^* \C^{2n}$. We show that this fixed-point set is formed
by $M(\alpha)$, the space of polygons in 
$\R^3$ obtained when $p=0$, and several other connected components
$Z_S$, where $S$ runs over all subsets of $\{1\, , \cdots \, ,n\}$ with
$\lvert S\rvert \geq 2$ and
\begin{equation}\label{eq:0short}
\sum_{i\in S} \alpha_i \,<\, \sum_{i\in S^c} \alpha_i
\end{equation}
(the complement of $S$ is denoted by $S^c$).
These components $Z_S$ are all non-compact except when $\lvert S \rvert \,=\,n-1$, in 
which case $Z_S\,=\,\C \P^{n-2}$ and $M(\alpha)$ is empty. Let $\mathcal{S}'(\alpha)$ be the collection of all subsets of $\{1\, , \cdots \, ,n\}$ with
$\lvert S\rvert \geq 2$ satisfying \eqref{eq:0short}.

We describe these sets $Z_S$ and the corresponding
components ${\mathcal Z}_S$ of the fixed point set of the involution of
${\mathcal H}(\beta)$ defined in \eqref{eq:0.1}; the following
theorem is proved (see Section~\ref{sec:inv}).

\begin{theorem}\label{thm:0.1}
The fixed-point set of the involution in \eqref{eq:0.1} of the space of 
parabolic Higgs bundles $\mathcal{H}(\beta)$ is
$$ \mathcal{H}(\beta)^{\Z/2\Z}\,=\, \mathcal{M}_{\beta,2,0} \cup \bigcup_{S \in 
\mathcal S'(\alpha)} \mathcal{Z}_S\, ,$$
with $\alpha_i=\beta_2(x_i)-\beta_1(x_i)$, where
$\mathcal{M}_{\beta,2,0}$ is the space of rank two degree zero
parabolic vector bundles over $\C{\mathbb P}^1$,  and where
$\mathcal{Z}_S\,\subset\, \mathcal{H}(\beta)$ is formed by parabolic Higgs 
bundles $\mathbf{E}\,=\,(E,\Phi)\in  \mathcal{H}(\beta)$ such that
\begin{enumerate}
\item[(i)] the parabolic vector bundle $E$ admits 
a direct sum decomposition $E\,=\,L_0\oplus L_1$,
where $L_0$ and $L_1$ are parabolic line bundles where the
parabolic weight of $L_0$ (respectively, $L_1$) at
$x_i\,\in\, S^c$ is $\beta_2(x_i)$ (respectively, $\beta_1(x_i)$), and
the parabolic weight of $L_0$ (respectively, $L_1$) at 
$x_i\,\in\, S$ is $\beta_1(x_i)$ (respectively, $\beta_2(x_i)$);

\item[(ii)] the residues of the Higgs field $\Phi$ at the parabolic points $x_i$ 
are either upper or lower triangular with respect to the above decomposition, according to whether $i$ is in $S$ or in $S^c$. 
\end{enumerate}
Moreover, $\mathcal{Z}_S$ is a non-compact manifold of dimension $2(n-3)$ except 
when $\lvert S\rvert \,= \,n-1$, in which case $\mathcal{Z}_S\,=\,
\mathcal{M}_S$ is compact and diffeomorphic to $\C \P^{n-3}$. In all cases, $\mathcal{H}(\beta)^{\Z/2\Z}$ has $2^{(n-1)}-(n+1)$ non-compact components and one compact component.
\end{theorem}

\begin{rem}
\mbox{}
\begin{itemize}
\item {\rm Since the vector bundle underlying $E$ is holomorphically trivial, it follows that the
holomorphic line bundles underlying $L_0$ and $L_1$ are both holomorphically
trivial. }

\item {\rm Statement (i) in Theorem \ref{thm:0.1}  means that if 
\begin{align*}
 E_{x_i,1} & \supset E_{x_i,2} \supset 0\\ 
0 \leq \beta_{1}(x_i) & < \beta_{2}(x_i) < 1
\end{align*}
is the parabolic structure, then
$E_{x_i,2}=E_{x_j,2}$ whenever $i,j \,\in\, S$ or $i,j \,\in\,
S^c$. Note that this condition is independent of the choice of the
trivialization of $E$.}
\end{itemize}
\end{rem}

In Section~\ref{sec:Mink}, we show that for any
$S\,\in \mathcal S'(\alpha)$, the corresponding component of the fixed 
point sets of the involution of $X(\alpha)$ (or of $\mathcal{H}(\beta)$) is 
diffeomorphic to a moduli space 
of polygons in Minkowski $3$-space, meaning $\R^3$ equipped with the Minkowski inner 
product 
$$
v \circ w \,=\, -x_1x_2-y_1y_2+t_1t_2\, ,
$$
for $v=(x_1,y_1,t_1)$ and $w\,=\,(x_2,y_2,t_2)$.
The surface $S_R$ in $\R^3$ defined by the equation $-x^2-y^2+t^2=R^2$ (a 
\emph{pseudosphere} of radius $R$) has two connected components: $S_R^+$,
corresponding to $t>0$, which is called a \emph{future pseudosphere}, and
$S_R^-$, 
corresponding to $t<0$, which is called a \emph{past pseudosphere}. The group 
${\rm SU}(1,1)$ acts transitively on $S_R^+$ (respectively,
$S_R^-$) since one can think of $\R^3$ 
as $\mathfrak{su}(1,1)^*$ with
$S_R^+$ (respectively, $S_R^-$) 
being an elliptic coadjoint orbit. Consequently, both
$S_R^+$ and $S_R^-$ have the ${\rm SU}(1,1)$--invariant Kostant--Kirillov symplectic structure
of a coadjoint orbit. Fixing 
two positive integers $k_1,k_2$ with $k_1+k_2\,=\,n$, we consider closed 
polygons in Minkowski $3$-space with the first $k_1$ sides lying in future 
pseudospheres of radii $\alpha_1, \cdots,\alpha_{k_1}$ and the last $k_2$ 
sides lying in past pseudospheres of radii $\alpha_{k_1+1}, \cdots ,
\alpha_n$. The space of all such closed polygons can be 
identified with the zero level set of the moment map
\begin{equation*}
\mu\,:\,\mathcal{O}_1\times \cdots \times \mathcal{O}_n\,\longrightarrow\,
\mathfrak{su}(1,1)^*
\end{equation*}
for the diagonal ${\rm SU}(1,1)$-action, where the coadjoint ${\rm SU}
(1,1)$--orbit $\mathcal{O}_i\,\cong\, S_{\alpha_i}^+$, $1 \leq i \leq k_1$,
is a future pseudosphere of radius $\alpha_i$, and
$\mathcal{O}_i\,\cong\, S_{\alpha_i}^-$, $k_1+1 \leq i \leq n$, is a
past pseudosphere of radius $\alpha_i$, equipped with its Kostant--Kirillov
symplectic structure \cite{F}. Then the corresponding moduli space of polygons is defined as the symplectic quotient
$$
M^{k_1,k_2}(\alpha)\,:=\,\mu^{-1}(0)/{\rm SU}(1,1)\, .
$$
We have the following result.

\begin{theorem}
For any $S\,\in\, \mathcal{S}^\prime(\alpha)$, the components $\mathcal{Z}_S$ 
and $Z_S$, of the fixed-point sets of the involutions in \eqref{eq:0.1} and 
\eqref{eq:0.2} respectively, are diffeomorphic to the moduli space 
$$
M^{\lvert S\rvert,\lvert S^c \rvert}(\alpha)
$$
of closed polygons in Minkowski $3$-space.
\end{theorem}

This interpretation allows us to see the fixed-point set of the above 
involutions as a moduli space of another related problem, thus helping us to 
understand many of its geometrical properties as seen in the example of Section~\ref{sec:ex}.

\medskip
\noindent
\textbf{Acknowledgements.}\, We thank O. Garc\'{i}a-Prada for suggesting the study
of the hyperpolygon description of the fixed-point set of
the natural involution of the moduli space of parabolic Higgs bundles.

\section{Hyperpolygon spaces} \label{hyperpolygons}

Let $\mathcal Q$ be the star-shaped quiver with vertices parametrized by $I\cup 
\{0\} \,=\, \{1\, ,\cdots\, ,n\}\cup \{0\}$ and the arrows parametrized by $I$ 
such that, for any $i\, \in\, I$, the tail and the head of the corresponding 
arrow are $i$ and $0$ respectively. Consider all representations of $\mathcal 
Q$ with $V_i \,=\, \C $, for $i \,\in\, I$, and $V_0\,=\, 
\C^2$. They are parametrized by
$$
E(\mathcal Q, V)\,:=\, \bigoplus_{i\in I}
\textrm{Hom} (V_i, V_0) \,=\, \C^{2n}\, .
$$
Using the actions of ${\rm U}(1)$ and ${\rm U}(2)$ on $\C$ and $\C^2$ 
respectively, we construct an action of ${\rm U}(2)\times{\rm U}(1)^n$ on
$E(\mathcal Q, V)$. This action produces an action of ${\rm U}(2)\times{\rm 
U}(1)^n$ on the cotangent bundle $T^*E(\mathcal Q, V)\,=\, T^*\C^{2n}$. One 
gets a hyper-K\"ahler quiver variety by performing the hyper-K\"ahler reduction 
on $T^*E(\mathcal Q, V)$ for this action of ${\rm U}(2)\times{\rm U}(1)^n$. 
Since the diagonal circle
$$
\{(c\cdot \text{Id}_{\C^{2}}, c,\cdots ,c)\, \mid\, ~
|c|\,=\, 1\}\,\subset\, {\rm U}(2) \times {\rm U}(1)^n
$$
acts trivially on $T^*E(\mathcal Q, V)$, the action factors through the 
quotient group
$$K \,:=\, \Big( {\rm U}(2) \times {\rm U}(1)^n\Big)/{\rm U}
(1)\,=\, \Big({\rm SU}(2) \times 
{\rm U}(1)^n\Big)/(\Z /2\Z)\, ,$$
where $\Z /2\Z$ acts as multiplication by $-1$ on each factor.
As $T^* \C^2 \,=\, (\C^{2})^* \times \C^2$ can be identified with 
the space of quaternions, the cotangent bundle $T^*E(\mathcal Q, V) \,=\, T^* \C^{2n}$ has 
a natural hyper-K\"ahler structure (see for example \cite{K2, H}).
The hyper-K\"ahler quotient of $T^* \C^{2n}$ 
by $K $ can be explicitly described as follows. Let
 $(p,q)$ be coordinates on $T^* \C^{2n}$, where $p=(p_1, \cdots, p_n)$ 
is the $n$-tuple of row vectors $p_i =\left( \begin{array}{ll} a_i &
b_i\end{array}\right) \in (\C^2)^*$ and 
$q= (q_1, \cdots, q_n)$ is the $n$-tuple of column vectors 
$q_i =\Big( \begin{array}{c}
c_i \\ d_i 
\end{array} \Big) \in \C^2$. In terms of these coordinates,
the action of $K$ on $ T^*\C^{2n}$ is given by 
$$ (p,q) \cdot [A; e_1, \cdots, e_n]= 
\Big( (e_1^{-1}p_1 A, \cdots, e_n^{-1}p_n A), ( A^{-1} q_1 e_1, \cdots, A^{-1} 
q_n e_n ) \Big).$$
This action is hyper-Hamiltonian with hyper-K\"ahler moment map
$$ \mu_{HK}:= \mu_{\R} \oplus \mu_{\C} : T^* \C^{2n} \longrightarrow 
\big(\mathfrak{su}(2)^* \oplus (\R^n)^*\big) \oplus \big(\mathfrak{sl}(2, 
\C)^* \oplus (\C^n)^*\big)\, ,$$
\cite{K2}, where the real moment map $\mu_{\R}$ is given by 
\begin{equation} \label{real} 
\mu_{\R} (p,q)\,=\,\frac{\sqrt{-1}}{2} \sum_{i=1}^n (q_i q_i^* -p_i^* p_i )_0 
\oplus \Big(\frac{1}{2} (|q_1|^2 -|p_1|^2), \cdots, 
\frac{1}{2} (|q_n|^2 -|p_n|^2) \Big)\, ,
\end{equation} 
and the complex moment map $\mu_{\C}$ is given by
\begin{equation} \label{complex}
\mu_{\C} (p,q)\,=\,- \sum_{i=1}^n (q_i p_i)_0 \oplus (\sqrt{-1}p_1 q_1, 
\cdots,\sqrt{-1} p_n q_n)\, .
\end{equation}
The hyperpolygon space $X(\alpha)$ is then defined to be the hyper-K\"ahler 
quotient
\begin{equation}\label{icd}
X(\alpha)\,:=\,T^*\C^{2n} / \!\! / \!\!/ \!\!/_{\alpha} K\,:=\, 
\Big( \mu_{\R}^{-1} (0, \alpha) \cap \mu_{\C}^{-1} (0,0) \Big) / K
\end{equation}
for $\alpha\,=\,(\alpha_1, \cdots, \alpha_n )\,\in\,\R^n_+$.

An element $(p,q)\,\in\,T^*\C^{2n}$ is in $\mu_{\C}^{-1} (0,0)$ if 
and only if
$$
p_i q_i\,=\,0 \quad \text{and} \quad \sum_{i=1}^n (q_i p_i)_0\, =\, 0\, .
$$
In other words, an element $(p,q)$ of $T^*\C^{2n}$ is in $\mu_{\C}^{-1} 
(0,0)$ if and only if
\begin{equation}\label{complex1}
 a_i c_i + b_i d_i = 0
\end{equation}
and 
\begin{equation}\label{complex2}
\sum_{i=1}^n a_i c_i - b_i d_i =0, \quad
\sum_{i=1}^n a_i d_i =0, \quad
\sum_{i=1}^n b_i c_i =0\, .
\end{equation}

Similarly, $(p,q)$ is in $\mu_{\R}^{-1} (0, \alpha) $ if and only if 
$$
\frac{1}{2} \big(|q_i|^2 -|p_i|^2\big) \,=\, \alpha_i	\quad \text{and} \quad 
\sum_{i=1}^n \big(q_i q_i^* -p_i^* p_i \big)_0 \,=\,0\, ,
$$
i.e., if and only if
\begin{equation}\label{real1}
|c_i|^2 +|d_i|^2 - |a_i|^2 - |b_i|^2 \,=\, 2 \alpha_i
\end{equation}
and 
\begin{equation}\label{real2}
\sum_{i=1}^n |c_i|^2 - |a_i|^2 + |b_i|^2 - |d_i|^2 \,=\,0, \quad
\sum_{i=1}^n a_i \bar{b_i} - \bar{c_i} d_i \,=\,0\, .
\end{equation}

An element $\alpha\,=\,(\alpha_1, \cdots, \alpha_n )\,\in\,\R^n_+$ is said to 
be 
\emph{generic} if and only if 
\begin{equation}
\label{eq:epsilon}
\varepsilon_S(\alpha)\,:= \,\sum_{i \in S} \alpha_{i} - \sum_{i \in S^c} 
\alpha_{i} \,\neq\, 0
\end{equation}
for every subset $S\,\subset\,\{1, \cdots, n\}$. For a generic $\alpha$, the 
hyperpolygon space $X(\alpha)$ is a 
non-empty complex manifold of complex dimension $2(n-3)$ (see \cite{K2} for details).

Hyperpolygon spaces can be described from an algebro-geometric point of view as geometric
invariant theoretic quotients. To elaborate this, we need the 
stability criterion, developed by 
Nakajima \cite{N2,N3} for quiver varieties and adapted by Konno \cite{K2} to hyperpolygon spaces. We will recall this below.

Let $\alpha$ be generic. A subset $S\,\subset\,\{1, \cdots, n\}$ is called 
\emph{short} if
\begin{equation}\label{eq:short}
\varepsilon_S(\alpha) \,<\,0
\end{equation}
and \emph{long} otherwise (see \eqref{eq:epsilon} for the definition
of $\varepsilon_S(\alpha)$).
Given $(p,q)\,\in\, T^* \C^{2n}$ and a subset $S\,\subset\,\{1, \cdots, n\}$, we 
say that $S$ is \emph{straight} at $(p,q)$ if $q_i$ is proportional to $q_j$ 
for all $i, j \,\in\, S$.

\begin{theorem}[\cite{K2}]\label{alpha-stability}
Let $\alpha \in \R^n_{+}$ be generic. A point $(p,q)\,\in \,T^* \C^{2n}$ is 
\emph{$\alpha$-stable} 
if and only if the following two conditions hold:
\begin{itemize}
\item[(i)] $q_i\,\neq\, 0$ for all $i$, and
\item[(ii)] if $S\,\subset\, \{1, \cdots, n \}$ is straight at $(p,q)$ and $p_j 
\,=\,0$ for all $j \in S^c$, then $S$ is short.
\end{itemize}
\end{theorem}

\begin{rem}\label{maximalstraight}
{\rm Note that it is enough to verify (ii) in Theorem \ref{alpha-stability} for all 
maximal straight sets, 
that is, for those that are not contained in any other straight set at $(p,q)$.}
\end{rem}

Let $\mu_{\C}^{-1} (0,0)^{\alpha\text{-st}}$ denote the set of points in 
$\mu_{\C}^{-1} (0,0)$ that are 
$\alpha$-stable, and let $$K^{\C}\,:=\, ({\rm SL}(2,\C) \times (\C^*)^n) /(\Z/2\Z)$$
be the complexification of $K$.

\begin{proposition}[\cite{K2}]\label{git}
Let $\alpha\,\in\,\R^n_{+}$ be generic. Then
$$\mu_{HK}^{-1} \big( (0, \alpha),(0,0) \big) \,\subset\, \mu_{\C}^{-1} 
(0,0)^{\alpha\text{-{\rm st}}}\, ,$$
and there exists a natural bijection 
$$ \iota\,:\,\mu_{HK}^{-1} \big( (0, \alpha),(0,0) \big)/K\,\longrightarrow\,
\mu_{\C}^{-1} (0,0)^{\alpha\text{-{\rm st}}}/ K^{\C}.$$
\end{proposition}

{}From Proposition \ref{git} and the definition in \eqref{icd} it follows
that
$$X(\alpha) \,=\, \mu_{\C}^{-1} (0,0)^{\alpha\text{-st}} / K^{\C}\, . $$
Following \cite{HP}, we denote the elements in $\mu_{\C}^{-1} 
(0,0)^{\alpha\text{-st}}/ K^{\C}$ by 
$[p,q]_{\alpha\text{-st}}$, and denote by $[p,q]_{\R}$ the elements in 
$\mu_{HK}^{-1} 
\big( (0, \alpha),(0,0) \big)/K$, when we need to make an explicit use of one of the two constructions. In all other cases, we will simply write $[p,q]$ for a hyperpolygon in $X(\alpha)$. 

\subsection{A circle action}

Consider the $S^1$-action on $X(\alpha)$ defined by 
\begin{equation}\label{action}
\lambda \cdot [p,q] \,=\, [\lambda \, p, q]\, .
\end{equation} 
This action is Hamiltonian with respect to 
symplectic structure on $X(\alpha)$; the associated moment map 
$\phi\,:\,X(\alpha) \,\longrightarrow\, \R$ is given by
\begin{equation}\label{eq:phi}
\phi([p,q]) \,= \,\frac{1}{2}\sum_{i=1}^n |p_i|^2\, .
\end{equation}
This $\phi$ is a Morse-Bott function that is proper and bounded from
bellow. Following Konno \cite{K2}, let us consider $\mathcal S(\alpha)$,
namely the collection of short sets for $\alpha$, and its subset 
$$ \mathcal S' (\alpha)\,:=\, \big\{ S \subset \{1, \cdots,n \}\,\mid\, S 
\text{ is } \alpha\text{-short}, |S| \geq 2 \big\}\, . $$
For any $S\,\in\,\mathcal S'(\alpha)$, define
$$ X_S \,:=\, \big\{[p,q] \in X(\alpha)\,\mid\, S \text{ and } S^c\,
\text{ are straight at } (p,q) \text{ and }\, p_j=0\, \,
\forall\, \, j \in S^c \big\}\,.$$ 
Then the fixed-point set of the circle action on $X(\alpha)$ is
the following.
\begin{theorem}[\cite{K2}]\label{thm:fixedpoints}
The fixed point set for the $S^1$-action in \eqref{action} is
$$ X(\alpha)^{S^1}\,=\, M(\alpha) \cup \bigcup_{S \in \mathcal S'(\alpha)} 
X_S\, .$$
The fixed-point set component $X_S$ is diffeomorphic to $\C\P^{|S|-2}$, and it
has index $2(n-1-|S|)$. 
\end{theorem}

Let us now determine the isotropy weights of the circle action in \eqref{action} 
at different fixed points.

For $S\,\in\, \mathcal S'(\alpha)$, let us fix 
$[p^\prime,q^\prime]_{\alpha\text{-st}}\,\in\, X_S$. We may assume that for
each $i\,\in\, S$,
$$
q_i^\prime=\left( \begin{array}{c} c_i \\ 0 \end{array}\right)\, \text{ and }\quad p_i^\prime=\left(\begin{array}{cc} 0 & b_i \end{array}\right), \, \text{for $i \in S$} 
$$
and for each $i\,\in\, S^c$,
$$
\quad q_i^\prime=\left( \begin{array}{c} 0 \\ d_i \end{array}\right)
\, \text{ and } \quad p_i^\prime=\left(\begin{array}{cc} 0 & 0 \end{array}\right), \, \text{for $i \in S^c$}. 
$$
Moreover, we can assume that $S\,=\,\{1, \cdots, l \}$ and that 
$b_1,b_2\,\neq\, 
0$. Since $c_i,d_i \,\neq\, 0$ for all $i$, there exists a unique element $h 
\,\in\, K^\C$ such that $(p^\prime,q^\prime)h\,=\,(p^0,q^0)\,\in\, 
\mu^{-1}_\C(0,0)^{\alpha\text{-st}}$, where for each $i\,\in\, S$,
$$
q_i^0=\left( \begin{array}{c} 1\\ 0 \end{array}\right)\,
\text{ and } \quad p_i^0= \left(\begin{array}{cc} 0 & b_i^0\end{array}
\right)\, ,
$$
and for  $i\,\in\, S^c$,
$$
q_i^0=\left( \begin{array}{c} 0 \\ 1\end{array}\right)\,
\text{ and }
\quad p_i^0=\left(\begin{array}{cc} 0 & 0\end{array}\right)\, ,
$$
with $b_1^0\,=\,1$ and $b_2^0 \,\neq\, 0$. There exists an open 
neighborhood 
$U$ 
of $(p^0,q^0)$ in $T^*\C^{2n}$ such that for all $(p,q)\,\in\, U \cap 
\mu^{-1}_\C(0,0)$, there is a unique element $[A;e_1,\cdots,e_n]\,\in\, 
K^\C$ satisfying the conditions that
$$
A^{-1} q_i e_i = \left\{ \begin{array}{l l} \left(\begin{array}{cc} 1 & 0 \end{array}\right)^t, & \text{if}\,\, i=1 \\
\left(\begin{array}{cc} 1 & r_1 \end{array}\right)^t, & \text{if}\,\, i=2 \\
\left(\begin{array}{cc} 1 & w_i \end{array}\right)^t, & \text{if}\,\, i=3, 
\cdots, l \\
\left(\begin{array}{cc} w_i & 1 \end{array}\right)^t, & \text{if}\,\, 
i=l+1,\cdots,n-1\\
\left(\begin{array}{cc} 0& 1 \end{array}\right)^t, & \text{if}\,\, i=n \end{array}\right.
$$
and 
$$
e_i^{-1}p_i A = \left\{ \begin{array}{l l} \left(\begin{array}{cc} 0& 1 \end{array}\right), & \text{if}\,\, i=1 \\
\left(\begin{array}{cc} -r_1r_2 & r_2 \end{array}\right), &\text{if}\,\, i=2 \\
\left(\begin{array}{cc} -z_i w_i & z_i \end{array}\right), & \text{if}\,\, 
i=3, \cdots, l \\
\left(\begin{array}{cc} z_i & -z_i w_i \end{array}\right), & \text{if}\,\, 
i=l+1,\cdots,n-1\\
\left(\begin{array}{cc} r_3 & 0 \end{array}\right), & \text{if}\,\, i=n, \end{array}\right.
$$
where $r_1,r_2,r_3$ are uniquely determined by 
\begin{equation}
\label{eq:coord1}
\{z_i,w_i\,\mid\, i=3,\cdots,n-1\}\, ;
\end{equation} 
so the functions in \eqref{eq:coord1} define
a local coordinate system in $X(\alpha)$ around
$[p^\prime,q^\prime]_{\alpha\text{-st}}$. Indeed, 
$\alpha$-stability is an open condition so that any $(p,q)\,\in\,
\mu_{\C}^{-1}(0,0)^{\alpha\text{-st}}$ sufficiently close to 
$(p^\prime,q^\prime)$ will be $\alpha$-stable. Moreover, there exist unique,
up to multiplication by $\pm I$, $$e_1,e_n\,\in\, \C\setminus\{0\}\,
\, \text{ and }\, \, A\,\in\, {\rm SL}(2,\C)\, $$ such that
$$
A^{-1}q_1 e_1 = \left(\begin{array}{c} 1 \\ 0 \end{array}\right), \quad \quad A^{-1}q_n e_n = \left(\begin{array}{c} 0 \\ 1 \end{array}\right)
$$
and 
$$
e_1^{-1}p_1 A \,= \left(\begin{array}{cc} 0 & 1 \end{array}\right).
$$
Then one can uniquely determine $e_2,\cdots,e_{n-1}$ such that
$$
A^{-1} q_i e_i = \left(\begin{array}{c} 1 \\ * \end{array}\right), \text{ for 
$i=2,\cdots, l$}, \quad \quad A^{-1}q_i e_i = \left(\begin{array}{c} * \\ 1 
\end{array}\right), \text{ for $i=l+1,\cdots, n-1$}.
$$

Now it can be easily shown that the $S^1$-action (constructed in 
\eqref{action}) in these local coordinates is given by
\begin{equation}
\label{eq:loccirc1}
\lambda \cdot (z_i,w_i)= \left\{ \begin{array}{ll} (z_i, \lambda w_i), & 
\text{if $i=3,\cdots,l$} \\ \\ (\lambda^2z_i, \lambda^{-1}w_i), & \text{if 
$i=l+1,\cdots, n-1$}. \end{array}\right.
\end{equation}

Let us now consider a fixed point $[0,q^\prime]_{\alpha\text{-st}}\,\in\,
M(\alpha)$. Then we may assume that 
 \begin{align*}
q^\prime_1 & = \left( \begin{array}{c} c_1 \\ 0 \end{array}\right), \,\, \text{with}\,\, c_1\neq 0, \quad q^\prime_2= \left( \begin{array}{c} c_2 \\ d_2 \end{array}\right), \, \text{with}\,\, c_2,d_2\neq 0, \\
q_3^\prime &=\left( \begin{array}{c} c_3 \\ d_3 \end{array}\right), \,\text{with} \,\, d_3\neq 0, \\
q_i^\prime & =\left( \begin{array}{c} c_i \\ d_i \end{array}\right), \, 
\text{with} \,\, c_i\neq 0, \,\, \text{for} \,\, i=4,\cdots,n-1, \\
 q_n^\prime & =\left( \begin{array}{c} 0 \\ d_n \end{array}\right), \, \text{with}\,\, d_n\neq 0,
\end{align*}
since $[0,q^\prime]_{\alpha\text{-st}}$ is not in any of the sets $X_S$.
As $c_i\,\neq\, 0$ for all $i\neq 3,n$, and $d_3\,\neq\, 0\, \neq\, d_n$, there 
exists a 
unique element $h\,\in\, K^\C$ such that $(0,q^\prime)h\,=\,(0,q^0)\,\in\, 
\mu^{-1}_\C(0,0)^{\alpha\text{-st}}$, where
 \begin{align*}
q^0_1 & = \left( \begin{array}{c} 1 \\ 0 \end{array}\right), \quad q^0_2= \left( \begin{array}{c}1 \\ 1 \end{array}\right), \quad
q_3^0 =\left( \begin{array}{c} w_3 \\ 1 \end{array}\right), \\
q_i^0 & =\left( \begin{array}{c} 1 \\ w_i \end{array}\right), \,\, \text{for} 
\,\, i=4,\cdots,n-1, \\
 q_n^0 & =\left( \begin{array}{c} 0 \\ 1 \end{array}\right). 
\end{align*}
Then there exists an open neighborhood $U$ of $(0,q^0)$ in $T^*\C^{2n}$ such 
that for all $(p,q)\in U \cap \mu^{-1}_\C(0,0)$, there is a unique element 
$[A;e_1,\cdots, e_n]\in K^\C$ such that
$$
A^{-1} q_i e_i = \left\{ \begin{array}{l l} 
\left(\begin{array}{cc} 1 & 0 \end{array}\right)^t, & \text{if}\,\, i=1 \\
\left(\begin{array}{cc} 1 & 1 \end{array}\right)^t, & \text{if}\,\, i=2 \\
\left(\begin{array}{cc} w_3 & 1 \end{array}\right)^t, & \text{if}\,\, i=3 \\
\left(\begin{array}{cc} 1 & w_i \end{array}\right)^t, & \text{if}\,\, i=4, 
\cdots, n-1 \\
\left(\begin{array}{cc} 0& 1 \end{array}\right)^t, & \text{if}\,\, i=n \end{array}\right.
$$
and 
$$
e_i^{-1}p_i A = \left\{ \begin{array}{l l}
\left(\begin{array}{cc} 0& r_1 \end{array}\right), & \text{if}\,\, i=1 \\
\left(\begin{array}{cc} -r_2 & r_2 \end{array}\right), & \text{if}\,\, i=2 \\
\left(\begin{array}{cc} z_3 & - z_3 w_3 \end{array}\right), &
\text{if}\,\, i=3, \cdots, l \\
\left(\begin{array}{cc} -z_i w_i & z_i \end{array}\right), & \text{if}\,\, 
i=4,\cdots,n-1\\
\left(\begin{array}{cc} r_3 & 0 \end{array}\right), & \text{if}\,\, i=n, \end{array}\right.
$$
where $r_1,r_2,r_3$ are uniquely determined by 
\begin{equation}
\label{eq:coord}
\{z_i,w_i\mid i=3,\cdots,n-1\}\, ;
\end{equation}
so \eqref{eq:coord} is a local coordinate system around
$[0,q^\prime]_{\alpha\text{-st}}$ in $X(\alpha)$. Indeed, $\alpha$-stability is 
an open condition and, moreover, there exist unique, up to multiplication by 
$\pm I$,
$$A\,\in\, {\rm SL}(2,\C) \,\, \text{ and }\,\, e_1,e_2,e_n\,\in\,
\C\setminus\{0\}$$ such 
that
$$
A^{-1}q_1 e_1 = \left(\begin{array}{c} 1 \\ 0 \end{array}\right), \quad A^{-1}q_2 
e_2 = \left(\begin{array}{c} 1 \\ 1\end{array}\right), \quad \text{and} \quad A^{-1}q_n e_n = \left(\begin{array}{c} 0 \\ 1 \end{array}\right).
$$
Then one can uniquely determine $e_3,\cdots,e_{n-1}$ such that
 $$
A^{-1}q_3 e_3 = \left(\begin{array}{c} * \\ 1 \end{array}\right) \quad
\text{and} \quad A^{-1} q_i e_i = \left(\begin{array}{c} 1 \\ * 
\end{array}\right), \text{ for $i\,=\,4,\cdots, n-1$}.
$$
It is straight-forward to check that the circle action (see \eqref{action}) in these 
local coordinates is given by
\begin{equation}
\label{eq:loccirc2}
\lambda \cdot (z_i,w_i)\,= \,(\lambda z_i, w_i)\,\,\text{for}\,\, 
i\,=\,3,\cdots,n-1\, .
\end{equation}
Using \eqref{eq:loccirc1} and \eqref{eq:loccirc2} we obtain the following result.

\begin{theorem}\label{thm:weights}
Let $[p,q]_{\alpha\text{-{\rm st}}}$ be a point in $X_S$. Then the non-zero
isotropy weights of the $S^1$-representation on
$T_{[p,q]_{\alpha\text{-{\rm st}}}} X(\alpha)$ are
\begin{itemize}
\item $+1$ with multiplicity $\lvert S \rvert - 2$;
\item $-1$ with multiplicity $(n-1)-\lvert S \rvert$;
\item $+2$ with multiplicity $(n-1)- \lvert S \rvert$.
\end{itemize}
Let $[0,q]_{\alpha\text{-{\rm st}}}$ be a point of
the space $M(\alpha)$. Then the
non-zero isotropy weights of the $S^1$--representation on
$T_{[0,q]_{\alpha\text{-{\rm st}}}} X(\alpha)$ are
\begin{itemize}
\item $+1$ with multiplicity $(n-1)-\lvert S \rvert$.
\end{itemize}
\end{theorem}

\section{Spaces of Parabolic Higgs bundles}\label{PHBs}

Let $\Sigma$ be $\C \P^1$ with $n$ ordered marked points $D\, =\,
\{x_1,\cdots,x_n\}$ and let $E$ 
be a parabolic vector bundle of rank two over $\Sigma$ with parabolic structure
\begin{align*}
& E_x \, :=\, E_{x,1}\,\supset\, E_{x,2}\,\supset\, 0\, , \\
0 \leq & \beta_1(x) < \beta_2(x) <1
\end{align*}
over each point of $D$. Its \emph{parabolic degree} is 
then
$$
\text{par-deg}(E)\,:=\,{\rm degree}(E) + \sum_{x \in D} \left( \beta_1(x) + 
\beta_2(x)\right)\, .
$$
We recall that $E$ is said to be \emph{stable} if 
$\text{par-}\mu(E)\,>\,\text{par-}\mu(L)$ 
for every line subbundle $L$ of $E$ equipped with the
induced parabolic structure, where, for any parabolic vector bundle $F$, the 
\emph{slope} $\text{par-}\mu(F)$ is defined as $\text{par-deg}(F)/\text{rank}(F)$.

Now if $L$ is a parabolic line subbundle of $E$, its induced parabolic structure is given 
by the trivial flag over each point $x$ of $D$,
$$
L_{x} \supset 0,
$$
with weights 
$$
\beta^L(x) = \left\{ \begin{array}{l} \beta_1(x), \quad \text{if} \quad L_x \cap E_{x,2} = \{ 0 \}, \\ \\ \beta_2(x), \quad \text{if} \quad L_x \cap E_{x,2} = \mathbb{C}, \end{array} \right.
$$
and so it has parabolic degree
$$
\text{par-deg}(L)\,=\,{\rm degree}(L) + 
\sum_{i\in S_L} \beta_2(x_i) + \sum_{i\in S_L^c} \beta_1(x_i),
$$
where 
\begin{equation}\label{eq:setpb}
S_L:=\{ i \in \{1, \cdots, n\} \mid\, \beta^L(x_i)=\beta_2(x_i)\}.
\end{equation}
Hence, $E$ is stable if and only if every parabolic line subbundle $L$ 
satisfies the inequality
\begin{equation}
\label{eq:stable}
{\rm degree}(E) - 2\cdot {\rm degree}(L) \,>\, \sum_{i \in S_L} 
\big(\beta_2(x_i) -\beta_1(x_i)\big) - \sum_{i \in S_L^c} \big(\beta_2(x_i) -\beta_1(x_i)\big).
\end{equation}

The holomorphic cotangent bundle of the Riemann surface $\Sigma$ will be 
denoted by $K_\Sigma$. The line bundle on $\Sigma$ defined by the divisor $D$ 
will be denoted by ${\mathcal O}_{\Sigma}(D)$. A \emph{parabolic Higgs bundle}
of rank two is a pair ${\bf E} \,:=\, (E, \Phi)$, where $E$ is a parabolic 
vector bundle over $\Sigma$ of rank two, and 
$$\Phi \in H^0(\Sigma, S Par End(E) \otimes K_\Sigma(D))$$ 
is a \emph{Higgs field} on $E$. Here $S Par End(E)$ denotes the 
subsheaf of $End(E)$ formed by strongly parabolic endomorphisms $\varphi\,:\,E 
\,\longrightarrow\, E$, which, in this situation, simply means that 
$$
\varphi(E_{x,1}) \subset E_{x,2} \quad \text{and} \quad \varphi(E_{x,2})=0, \quad \text{for all $x\in D$}.
$$ 
Note that $\Phi$ is then a meromorphic endomorphism-valued one-form with 
simple poles along $D$ whose residue at each $x\,\in\, D$ is nilpotent with respect to the 
flag, i.e., 
$$
(\text{Res}_x \Phi) (E_{x,i}) \,\subset\, E_{x,i+1}
$$
for all $i=1,2$ and $x\in D$, with $E_{x,3}=0$. The definition of stability
extends to Higgs bundles: a parabolic Higgs bundle ${\bf E}= (E,\Phi)$ is 
\emph{stable} if
$\text{par-}\mu(E) \,>\, \text{par-}\mu(L)$ for all parabolic line 
subbundles $L\subset E$ which are preserved by $\Phi$.

Let $\mathcal{H}(\beta)$ be the moduli space of parabolic Higgs bundles of rank 
two such that the underlying holomorphic vector bundle is holomorphically
trivial. In \cite{GM} it is shown that $\mathcal{H}(\beta)$ is diffeomorphic to
the space of hyperpolygons $X(\alpha)$ with $\alpha_i=\beta_2(x_i)-\beta_1(x_i)$. The correspondence between these two spaces is given by the map
\begin{equation}
\label{eq:isom}
\begin{array}{rl}
\mathcal{I}:X(\alpha) \longrightarrow & \mathcal H(\beta)\\ \\
{[p,q]}_{\alpha \textnormal{-st}} \longmapsto & (E_{(p,q)}\, , {\Phi}_{(p,q)}) =: 
\mathbf E_{(p,q)} \\
\end{array}
\end{equation}
where $E_{(p,q)}$ is the trivial vector bundle $\C\P^1\times \C^2
\,\longrightarrow\, \C\P^1$ with the parabolic structure consisting of weighted flags 
\begin{align*}\label{flag}
 \C^2 & \supset \langle q_i\rangle \supset 0\\ 
0 \leq \beta_{1}(x_i) & < \beta_{2}(x_i) < 1
\end{align*}
over the $n$ marked points $\{x_1,\cdots,x_n\} \,=\, D \,\subset\, \C \P^1$ with $\beta_i(x_j)$ 
satisfying
\begin{equation}\label{dif}
\beta_2(x_i)-\beta_1(x_i)\,=\,\alpha_i\, ,
\end{equation}
and $\Phi_{[p,q]} \,\in\, 
H^0\big(S Par End(E_{(p,q)}) \otimes K_{\C\P^1}(D)\big)$
is the Higgs field uniquely determined by the following condition
on the residue:
\begin{equation}
\label{eq:res}
\text{Res}_{x_i} \Phi \,:=\, (q_i p_i)_0
\end{equation}
at each $x_i\,\in\, D$.
In particular, the polygon space $M(\alpha)$ (obtained when $p=0$) is mapped to 
the moduli space $\mathcal{M}_{\beta,2,0}$ of parabolic vector
bundles of rank two over $\Sigma$ such that the underlying holomorphic
vector bundle is trivial (this map is obtained by setting $\Phi=0$). 

This isomorphism is equivariant with respect to the circle
action on $X(\alpha)$ (see \eqref{action}) and the circle action on $\mathcal{H}(\beta)$ defined by
\begin{equation}
\label{eq:actionHiggs}
\lambda \cdot (E,\Phi) = (E, \lambda \Phi), \quad \text{for $\lambda \in S^1$}.
\end{equation}
Each connected component $X_S$ of the fixed point set of
the circle action on $X(\alpha)$ is mapped to a manifold $\mathcal{M}_S$
formed by the trivial holomorphic bundle $E$ over $\Sigma$ equipped with
weighted flag structures 
\begin{align*}
 \C^2 & \supset E_{x_i,2} \supset 0\\ 
0 \leq \beta_{1}(x_i) & < \beta_{2}(x_i) < 1
\end{align*}
such that $E_{x_i,2}\,=\,E_{x_j,2}$ whenever $i,j \in S$ or $i,j \in S^c$, and an 
Higgs 
field with zero residue at all points $x_i$ with $i \in S^c$. Note that this 
description of the critical sets agrees with the one given by Simpson in 
\cite{Si}. Indeed, the bundles in $\mathcal{M}_S$ have a direct sum 
decomposition $E\,=\,L_0\oplus L_1$ as parabolic bundles, where the
parabolic weight of $L_0$ (respectively, $L_1$) at
$x_i\,\in\, S^c$ is $\beta_2(x_i)$ (respectively, $\beta_1(x_i)$), and
the parabolic weight of $L_0$ (respectively, $L_1$) at 
$x_i\,\in\, S$ is $\beta_1(x_i)$ (respectively, $\beta_2(x_i)$). (Note that
the holomorphic line bundles underlying $L_0$ and $L_1$ are trivial.)
Moreover, $\Phi$, being lower triangular with 
respect to this decomposition, preserves $L_1$ and induces a nonzero strongly 
parabolic homomorphism $\Phi_{|_{L_0}}\,:\,L_0\,\longrightarrow\, L_1\otimes 
K_\Sigma(D)$; that this strongly parabolic homomorphism is nonzero
follows from the fact that $X_S \cap M(\alpha) \,=\,\emptyset$. 

\section{An involution}\label{sec:inv}

As before, let $\mathcal{H}(\beta)$ be the moduli space of $\beta$-stable 
parabolic Higgs bundles of rank two such that the underlying holomorphic vector 
bundle is trivial. In this section we will 
restrict ourselves to the action of $\Z/2\Z\, \subset\, S^1$ on
$\mathcal{H}(\beta)$ giving the involution
\begin{equation}\label{eq:invPHB}
(E, \Phi) \,\longmapsto\, (E, -\Phi)\, ,
\end{equation}
and we will study the fixed-point set of this involution.

The parabolic Higgs bundles with zero Higgs field are clearly fixed by the 
involution in \eqref{eq:invPHB}, and so the moduli space 
$\mathcal{M}_{\beta,2,0}$ of $\beta$-stable rank two holomorphically trivial 
parabolic vector bundles over $\C \P^1$ is contained in the fixed point set of the 
involution. 

For the remaining fixed points, we will use the isomorphism in
\eqref{eq:isom} and study the fixed point set of the corresponding involution 
on the hyperpolygon space $X(\alpha)$ with 
$\alpha\,=\,(\alpha_1,\cdots,\alpha_n)$ satisfying \eqref{dif}. The fixed-point 
set of the involution
\begin{equation}
\label{eq:inv}
(p,q) \,\longmapsto\, (-p,q)\, ,
\end{equation}
on the hyperpolygon space $X(\alpha)$ is the set of points $X(\alpha)^{\Z/2\Z}$
that are fixed by the action of ${\Z/2\Z}\, \subset\, S^1$ in \eqref{action}.

As before, $M(\alpha)$ denotes the moduli space of polygons in $\R^3$.
Theorem \ref{thm:invfixedpoints} describes the fixed-point set of the
$\Z/2\Z$-action in \eqref{eq:inv}.

For each element $S$ of $\mathcal S'(\alpha)$, 
$$ Z_S \,:= \,\big\{[p,q] \in X(\alpha) \,\mid\,~ S \, \text{ and } S^c\, 
\text{ are straight at }\, (p,q) \big\}\,.$$ 

\begin{theorem}
\label{thm:invfixedpoints}
The fixed-point set of the involution in \eqref{eq:inv} is
$$ X(\alpha)^{\Z/2\Z}\,=\, M(\alpha) \cup \bigcup_{S \in \mathcal S'(\alpha)} 
Z_S\, ,$$
where $Z_S$ is defined above.

Moreover, $Z_S$ is a non-compact manifold of dimension $2(n-3)$ except when 
$\lvert S\rvert = n-1$, in which case $Z_S=X_S$ is compact and diffeomorphic
to $\C \P^{n-3}$. In all cases, $X(\alpha)^{\Z/2\Z}$ has $2^{n-1}-(n+1)$ non-compact components and one compact component.
\end{theorem}

\begin{proof}
{}From the isotropy weights of the $S^1$--action given in Theorem
\ref{thm:weights} it follows immediately that, if $M(\alpha)$ is nonempty, then
it is a connected component of the fixed point set of the $\Z_2/2\Z$--action. 
Furthermore, the connected components of the complement $X(\alpha)^{\Z/2\Z} 
\setminus M(\alpha)$ are parametrized by the elements $S$ of $S'(\alpha)$ and have dimension
$$
\dim Z_S \,=\, 2((n-1)-\lvert S \rvert )+ \dim X_S \,=\, 2(n-3)\, .
$$ 
Therefore, it remains to show that each connected component $Z_S$ of 
$X(\alpha)^{\Z/2\Z}\setminus M(\alpha)$ can be described as
$$ Z_S \,=\, \big\{[p,q] \in X(\alpha) \mid S \text{ and } S^c\, \text{are 
straight at }\, (p,q) \big\}\, .$$ 

Suppose that $[p,q]\,\in\, X(\alpha)^{\Z/2\Z}\setminus M(\alpha)$. Then there 
exists 
an element $$[A;e_1,\cdots,e_n]\,\in\, K\setminus \{I\}$$ such that
$$
e_i^{-1} p_i A \,=\, - p_i \quad \text{and}\quad A^{-1}q_i e_i \,=\,q_i, 
\,\,\text{for} \,\, i\,=\,1,\cdots, n\, ,
$$
and so
$$
A p_i^* \,=\, -e_i p_i^* \quad \text{and} \quad A q_i \,=\,e_i 
q_i\,\,\text{for} 
\,\, 
i\,=\,1,\cdots, n\, .
$$
Since $\lvert q_i\rvert^2 - \lvert p_i \rvert^2=2\alpha_i$, we have $q_i\neq 0$ 
for all $i=1,\cdots,n$ and so $q_i$ is an eigenvector of $A$ with eigenvalue 
$e_i$. Moreover, since $[p,q]\in X(\alpha)^{\Z/2\Z}\setminus M(\alpha)$, there 
exists an integer $i_0\,\in\, \{1,\cdots, n \}$ such that $p_{i_0}\neq 0$, and so 
$p_{i_0}^*$ is 
an eigenvector of $A$ with eigenvalue $-e_{i_0}$. Hence, assuming that $A$ is diagonal, we have 
$$
A\,=\,\left( \begin{array}{cc} e_{i_0} & 0 \\ 0 & -e_{i_0} \end{array} 
\right)
$$
with $e_{i_0}=\pm \sqrt{-1}$. We conclude that there exists an index set 
$S\subset \{1,\cdots, n \}$ with $i_0 \in S$ such that
\begin{align} \label{eq:Z0} 
p_i & = \left( \begin{array}{ll} 0 & b_i \end{array}\right), \quad q_i= \left(\begin{array}{c} c_i \\ 0 \end{array} \right), \forall i\in S \\ \nonumber 
p_i & = \left( \begin{array}{ll} a_i & 0 \end{array}\right), \quad q_i= \left(\begin{array}{c} 0 \\ d_i \end{array} \right),\forall i \in S^c.
\end{align}
Since $\lvert q_i\rvert^2 - \lvert p_i\rvert^2=2\alpha_i$, we conclude that
\begin{equation}
\label{eq:Z}
\lvert c_i\rvert^2 - \lvert b_i\rvert^2=2\alpha_i\,\,\text{for all $i\in S$} \quad \text{and} \quad \lvert d_i\rvert^2 - \lvert a_i\rvert^2=2\alpha_i \,\, \text{for all $i\in S^c$}.
\end{equation}
Moreover, since $\sum_{i=1}^n ( q_i q_i^* - p_i^*p_i)_0 =0$, we obtain that
\begin{equation}
\label{eq:Z0nn}
\sum_{i\in S} \left(\lvert c_i\rvert^2 + \lvert b_i\rvert^2\right) - \sum_{i 
\in S^c} \left( \lvert d_i\rvert^2 + \lvert a_i\rvert^2\right)\,=\, 0 
\end{equation}
and so, using \eqref{eq:Z}, we get that
\begin{equation}\label{eq:Z2}
\sum_{i \in S} \alpha_i - \sum_{i\in S^c} \alpha_i \,=\,\sum_{i\in S^c}\lvert 
a_i\rvert^2 -\sum_{i\in S} \lvert b_i \rvert^2\, .
\end{equation}
On the other hand, since $\sum_{i=1}^n (q_ip_i)_0\,=\,0$, we have that 
\begin{equation}\label{eq:zero}
\sum_{i \in S} b_i c_i \,=\,\sum_{i \in S^c} a_i d_i \,=\, 0\, .
\end{equation}
If $S$ is short, then we work with $S$ and, in particular, since $c_i\neq 0$ and 
there exists an $i_0\in S$ such that $b_{i_0}\neq 0$, from \eqref{eq:zero} 
it follows that there exists another $i_1 \in S$ such that $b_{i_1}\neq 0$ and 
we obtain that $S$ has cardinality at least two. If, on the other hand, $S$ is 
long, we consider $S^c$ instead (which is now short). Since $S$ is long, from
\eqref{eq:Z2} it follows that there is at least one $i_1\,\in\, S^c$ such that 
$p_{i_1}\neq 0$ and then, since $d_{i_1}\neq 0$, from \eqref{eq:zero} it follows 
that there is 
another element $i_2$ in $S^c$ with $p_{i_2}\neq 0$, implying that the short set 
$S^c$ has cardinality at least two.

Finally, since for every subset $S\subset \{1,\cdots,n\}$, we have that either $S$ or $S^c$ is short, the number of short sets for $\alpha$ is
$$
\frac{1}{2} \sum_{k=1}^{n-1} \left( \begin{array}{l} n \\ k \end{array} \right) = 2^{(n-1)}-1.
$$
If there is no short set of cardinality $n-1$, then there are exactly $n$ short sets of cardinality $1$ and so
$$
\lvert \mathcal{S}^\prime(\alpha) \rvert = 2^{(n-1)}-(n+1).
$$
Moreover, in this case, all the components $Z_S$ are non-compact. If, on the other hand, there is a short set $\widetilde{S}$ of cardinality $n-1$, then there are only $n-1$ short sets of cardinality $1$ and then the number of elements in $\mathcal{S}^\prime(\alpha)$ is
$$
\lvert \mathcal{S}^\prime(\alpha) \rvert = 2^{(n-1)}-n.
$$
However, in this case, $M(\alpha)$ is empty and $Z_{\widetilde{S}}$ is compact. We conclude that, in both cases, the number of non-compact components of $X(\alpha)^{S^1}$ is $2^{(n-1)}-(n+1)$ and that there is exactly one compact component (which is either $M(\alpha)$ or $Z_{\widetilde{S}}$). 
\end{proof}

Each manifold $Z_S$, being a component of $X(\alpha)^{\Z/2\Z}$, is symplectic and 
invariant under the circle action in \eqref{action}. Hence, whenever $\lvert 
S\rvert \neq n-1$, we obtain an effective Hamiltonian circle action on $Z_S$
(the action factors through the  quotient of $S^1$ by $\Z/2\Z$). The corresponding moment map then 
coincides with the restriction of $\frac{1}{2}\phi$ to $Z_S$. The only 
critical submanifold of this map is $X_S$ where it attains its minimum value. Consequently, we have the following results.

\begin{theorem} \label{thm:defret}
Each manifold $X_S\,\cong\, \C \P^{\lvert S \rvert -2}$ is a deformation 
retraction of $Z_S$. In particular, $Z_S$ is simply connected.
\end{theorem}

\begin{theorem}\label{thm:Poinc}
Let $P_t(M)$ be the Poincar\'{e} polynomial of $M$. Then
$$
P_t(Z_S)\,=\,P_t(X_S)\,=\,P_t(\C \P^{\lvert S \rvert-2})\,=\, 1+t+\cdots + 
t^{2(\lvert S \rvert-2)}\, .
$$
\end{theorem}

Going back to the space of parabolic Higgs bundles $\mathcal{H}(\beta)$
and using the isomorphism in \eqref{eq:isom}, we obtain 
from \eqref{eq:res} and \eqref{eq:Z0} that the fixed-point set 
$\mathcal{H}(\beta)^{\Z/2\Z}$ of the involution in \eqref{eq:invPHB} is described 
as in Theorem~\ref{thm:0.1}.

\section{Polygons in Minkowski $3$-space}

Let us consider the \emph{Minkowski inner product} on $\R^3$
$$
v \circ w \,=\, -v_1w_1-v_2w_2+v_3w_3\, ,
$$
for $v\,=\,(v_1,v_2,v_3)$ and $w\,=\,(w_1,w_2,w_3)$.
The inner product space consisting of $\R^3$ together with this signature 
$(-,-,+)$-inner product will be denoted by $\R^{2,1}$; it is called 
the \emph{Minkowski $3$-space}. The \emph{Minkowski norm} of a vector $v\in 
\R^{2,1}$ is then defined to be 
$$
\lvert\lvert v\rvert\rvert_{2,1}\,=\, \sqrt{\lvert v\circ v \rvert}\, .
$$
All elements $v$ of $R^{2,1}$ are classified according to the sign of $v \circ v 
$. The set of all $v$ such that $v\circ v =0$ is called the \emph{light cone}
of $\R^3$; any vector $v$ with $v\circ v =0$ is said to be 
\emph{light-like}. If $ v\circ v \,>\, 0$, then 
$v$ is called \emph{time-like}, and if $v\circ v \,<\,0$, then it is 
called \emph{space-like}. A time-like vector is said to be lying in \emph{future}
(respectively, \emph{past}) if $v_3\,>\,0$ (respectively, $v_3<0$). Note that the 
exterior of the light cone consists of all space-like vectors, while its interior 
consists of all time-like vectors. From now on we will write any $v\,\in\,\R^{2,1}$ 
as $v\,=\,(x,y,t)$.

Moduli spaces of polygons in $\R^{2,1}$ were described by Foth, \cite{F}, as 
follows. Consider the surface $S_R$ in $\R^3$ defined by the equation 
$-x^2-y^2+t^2\,=\,R^2$, which is called a \emph{pseudosphere}. The 
Minkowski metric on $\R^{2,1}$ restricts to a constant curvature Riemannian metric on 
$S_R$. It is an hyperboloid of two sheets. The connected 
component $S_R^+\, \subset\, S_R$, corresponding to $t\,>\,0$, is called a 
\emph{future pseudosphere}, and the connected component $S_R^-\, \subset\, S_R$, corresponding to 
$t\,<\, 0$, is called a \emph{past 
pseudosphere}. The group ${\rm SU}(1,1)$ acts transitively on each connected 
component since one can think of $\R^{2,1}\,\cong\,\R^3$ as 
$\mathfrak{su}(1,1)^*$
with $S_R^+$ and $S_R^-$ being elliptic coadjoint orbits. Consequently, both $S_R^+$ and 
$S_R^-$ have a natural invariant symplectic structure (the Kostant--Kirillov
form on a coadjoint orbit). The Minkowski metric is 
also invariant (since ${\rm SU}(1,1)$ acts by isometries) and both connected 
components are K\"{a}hler manifolds; they are in fact isomorphic to the hyperbolic 
plane ${\rm SU}(1,1)/{\rm U}(1)$. We will study the geometry of the symplectic 
quotients of the products of several future and past pseudospheres
with respect to 
the diagonal ${\rm SU}(1,1)$--action.

Let $\alpha\,=\,(\alpha_1,\cdots,\alpha_n)$ be an $n$-tuple of 
positive real numbers. Let us fix two positive integers $k_1,k_2$ with 
$k_1+k_2\,=\,n$. We will consider polygons in Minkowski $3$-space that have the 
first $k_1$ edges in the future time-like cone and the last $k_2$ edges in the past 
time-like cone, such that the Minkowski length of the $i$-th edge is $\alpha_i$. A 
closed 
polygon will then be one whose sum of the first $k_1$ sides in the future time-like cone coincides with the negative of the sum of the last $k_2$ sides in the past time-like cone. The space of all such closed polygons can be identified with the zero level set of the moment map
\begin{equation}
\label{eq:mmhyp}
\begin{array}{ccc}
\mu: \mathcal{O}_1\times \cdots \times \mathcal{O}_n & \longrightarrow & 
\mathfrak{su}(1,1)^* \\
(u_1, \cdots, u_n) & \longmapsto & u_1 + \cdots + u_n
\end{array}
\end{equation}
for the diagonal ${\rm SU}(1,1)$--action, where $\mathcal{O}_i\cong 
S_{\alpha_i}^+$ 
is a future pseudosphere of radius $\alpha_i$ if $1 \leq i \leq k_1$, and 
$\mathcal{O}_i\cong S_{\alpha_i}^-$ is a past pseudosphere if $k_1+1 \leq i \leq 
n$, equipped with the Kostant-Kirillov symplectic form on coadjoint orbits. 
Hence,
$$
M^{k_1,k_2}(\alpha)\,:=\,\mu^{-1}(0)/{\rm SU}(1,1)\, ,
$$
which is a quotient of a non-compact space by a non-compact Lie group. For a 
generic choice of $\alpha$, meaning $M^{k_1,k_2}(\alpha)$ is 
non-empty with $\sum_{i=1}^{k_1}\alpha_i\,\neq\,\sum_{i=k_1+1}^{n}\alpha_i$, every 
point in $M^{k_1,k_2}(\alpha)$ represents a polygon with a trivial stabilizer. 
In that situation, the group ${\rm SU}(1,1)$ acts freely and properly on $\mu^{-1}(0)$. 
Moreover, $0$ is a regular value of the moment map $\mu$ and so the quotient space $M^{k_1,k_2}(\alpha)$ is, for a generic $\alpha$, a smooth symplectic manifold of dimension $2(n-3)$. Note that the spaces
$$
M^{k_1,k_2}(\alpha_1,\cdots,\alpha_{k_1},\alpha_{k_1+1},\cdots,\alpha_{n})\quad 
\text{and} \quad 
M^{k_2,k_1}(\alpha_{k_1+1},\cdots,\alpha_{n},\alpha_{1},\cdots,\alpha_{k_1})
$$
are symplectomorphic by the isomorphism given by the involution of
${\mathbb R}^{2,1}$ defined by $(x,y,t)\,\longmapsto\, (-x,-y,-t)$.

\begin{theorem}[\cite{F}]\label{thF}
The space $M^{k_1,k_2}(\alpha)$ is non-compact, unless $k_1=1$ or $k_2=1$ in which case it 
is compact.
\end{theorem}

We give a brief outline of an argument for Theorem \ref{thF}.

Let us first assume that $k_2=1$. The last side of the polygon can be 
represented (after being acted on by an element of ${\rm SU}(1,1)$) by a vector in 
$\R^{2,1}$ with coordinates $(0,0,-\alpha_n)$. Hence, the sum of the first $n-1$ future 
time-like sides of the polygon is  $(0,0,\alpha_n)$. The only symmetry 
left is the circle rotation around the $t$-axis. Therefore, this space
of polygons is clearly bounded 
and closed and therefore compact.

The space $M^{k_1,k_2}(\alpha)$ is non-compact if $k_2\,>\,1$. For 
example, let us consider the simple case where $k_1=k_2=2$ and 
$\alpha_1=\alpha_2=\alpha_4=1$. Again we can assume that the last side is $(0,0,-1)$ and the only symmetry left is the circle rotation around the $t$-axis. Let $x_n$ be the closed polygon with sides
\begin{align*}
u_1& =(-1,0,\sqrt{2}), \quad u_2=(1-P(n),Q(n),n-\sqrt{2}), \\ u_3 &=(P(n),-Q(n),1-n)\quad \text{and} \quad u_4=(0,0,-1),
\end{align*}
where 
$$
P(n)=\frac{1}{2}(3+2(\sqrt{2}-1)n) \quad \text{and} \quad Q(n)=\sqrt{8(\sqrt{2}-1)n^2-4(3\sqrt{2}-1)n-9} \, .
$$
The sequence $\{x_n\}$ has no limit point in $M^{2,2}(1,1,2,1)$ and so this 
space is not compact.

Let us now describe the symplectic structure on $M^{k_1,k_2}(\alpha)$. For that, 
define the \emph{Minkowski cross product} $\dot{\times}$ as
$$
v\dot{\times} w\, :=\, \det \left( \begin{array}{rrr} -e_1 & -e_2 & e_3 \\ v_1 & 
v_2 & v_3 \\ w_1 & w_2 & w_3 \end{array}\right),
$$
where $v=(v_1,v_2,v_3)$ and $w=(w_1,w_2,w_3)$ with $e_1,e_2,e_3$
being the standard 
unit vectors in $\R^3$. This cross product satisfies the usual properties:
\begin{align*}
& v \dot{\times} w \,=\, - w \dot{\times} v\\
 & (u\dot{\times} v) \dot{\times} w + (v \dot{\times} w)\dot{\times} u + (w 
\dot{\times} u)\dot{\times} v\,=\, 0 
\end{align*}
and so $(\R^3,\dot{\times})$ is a Lie algebra. Moreover, it is isomorphic to 
$\mathfrak{su}(1,1)$ via the map
$$
\left(\begin{array}{l} x \\ y \\ t \end{array} \right)\,\longmapsto\,\frac{1}{2} 
\left( \begin{array}{cc} -\sqrt{-1} t & x + \sqrt{-1} y \\ x - \sqrt{-1} y & 
\sqrt{-1} t \end{array} \right)\, .
$$
Under this identification, the Minkowski inner product $\circ$ corresponds to $(A,B)\, 
\longmapsto\, -2 \cdot \mathrm{trace}(AB)$. 

The symplectic form on the pseudosphere $S_R$ is then given by
$$
\omega_u(v,w)\,=\,\frac{1}{R^2}\,u \circ (v\dot{\times} w)\, ,
$$
where $u\,\in\, S_R$ and $v,w\,\in \,T_u S_R$ (we think of $T_u S_R$ as the 
linear 
subspace of $\R^{2,1}$ orthogonal to $u$ with respect to the Minkowski inner product), and the map 
in \eqref{eq:mmhyp} is the moment map for the diagonal ${\rm SU}(1,1)$-action and 
the product symplectic structure.

\section{Back to the involution}
\label{sec:Mink}

Let us go back to Section~\ref{sec:inv} and consider a point $[p,q]$ in some $Z_S$. Let $u_i\in \R^3$ be the vector
$$
u_i \,:=\, \frac{\sqrt{-1}}{2} (p^*_ip_i -q_i q^*_i)_0 +\frac{1}{2} (p_i^*q_i^* + q_i p_i)_0\, ,
$$
where we use the identifications, $\mathfrak{su}(2)^*\cong( \R^3)^* 
\cong \mathfrak{su}(1,1)$.
If $i \in S$ then $p_i=\left( \begin{array}{ll}0 & b_i\end{array}\right)$ and $q_i=\left( \begin{array}{l} c_i \\ 0 \end{array}\right)$, implying that
\begin{equation}
\label{eq:uS}
u_i\,=\,\left( \operatorname{Re}{(b_i c_i)}\, , \operatorname{Im}{(b_i 
c_i)}\, ,\frac{\lvert b_i\rvert^2+\lvert c_i \rvert^2}{2}\right)
\end{equation}
with
$$
 u_i\circ u_i \,=\, (\lvert c_i\rvert^2-\lvert b_i \rvert^2)^2/4= \alpha_i^2
$$
($u_i$ has Minkowski norm $\alpha_i$).
Similarly, if $i \in S^c$, we have $p_i=\left( \begin{array}{ll}a_i & 0 \end{array}\right)$ and $q_i=\left( \begin{array}{l} 0 \\ d_i \end{array}\right)$, yielding
\begin{equation}
\label{eq:uSc}
u_i=\left( \operatorname{Re}{(a_i d_i)},- \operatorname{Im}{(a_i d_i)},-\frac{\lvert a_i\rvert^2+\lvert d_i \rvert^2}{2}\right)
\end{equation}
and
$$
u_i\circ u_i \,=\, (\lvert a_i\rvert^2-\lvert d_i \rvert^2)^2/4\,=\,
\alpha_i^2\, .
$$
Moreover, by \eqref{eq:Z0nn} and \eqref{eq:zero} we have that
$$
\sum_{i=1}^n u_i\,=\,0\, .
$$
So the vectors $u_i$ form a closed polygon in Minkowski $3$-space with the first 
$\lvert S\rvert$ sides in the positive time-like cone and the last $n-\lvert 
S\rvert$ sides in the past, with the $i$-th side being of Minkowski length $\alpha_i$.

\begin{theorem}\label{thm:polMink}
For any $S\,\in\, \mathcal{S}^\prime(\alpha)$, the components $\mathcal{Z}_S$ 
and $Z_S$, of the fixed-point sets of the involutions in \eqref{eq:invPHB} and 
\eqref{eq:inv} respectively, are diffeomorphic to the moduli space 
$$
M^{\lvert S\rvert,\lvert S^c \rvert}(\alpha)
$$
of closed polygons in Minkowski $3$-space.
\end{theorem}

\begin{proof} Let $S$ be a short set of cardinality at least two. 
Consider the map $\varphi:Z_S \longrightarrow M^{\lvert S\rvert,\lvert S^c 
\rvert}(\alpha)$ defined above, that is, $\varphi([p,q]_\R)$ is the element of 
$M^{\lvert S\rvert,\lvert S^c \rvert}(\alpha)$ represented by the polygon whose
sides are the vectors $u_i$ given by \eqref{eq:uS} and \eqref{eq:uSc} 
for $i$ in $S$ and $S^c$ respectively. 

Note that the pseudo-unitary group ${\rm SU}(1,1)$ is generated by the following
orientation preserving isometries of the pseudosphere: $A_\theta$ 
and $T_\phi$, where
$$
A_\theta =\left(\begin{array}{ccc} \cos{\theta} & -\sin{\theta} & 0 \\ \sin{\theta} & \cos{\theta} & 0 \\ 0 &0 & 1\end{array} \right),
$$
is an Euclidean \emph{rotation} by an angle $\theta$ in the $(x,y)$-plane, and
$$
T_\phi =\left(\begin{array}{ccc} 1 & 0 & 0 \\ 0 & \cosh{\phi} & \sinh{\phi} \\ 0 & \sinh{\phi} & \cosh{\phi} \end{array} \right)
$$
is a \emph{boost} of rapidity $\phi$ along the $y$-direction\footnote{In special relativity, the rapidity parameter $\phi$ is defined by $\tanh{\phi}=v/c$,
where $v$ is the velocity.} (cf. \cite{BV} for the details).

Let us first see that $\varphi$ is well defined. For that, consider two 
representatives $(p,q)$ and $(p^\prime,q^\prime)$ of the same element $[p,q]_\R$ in 
$Z_S$. Then there exists $[A;e_1,\cdots,e_n]\,\in\, K$ such that
$$
e_i^{-1}p_i A= p_i ^\prime \quad \text{and} \quad A^{-1} q_i e_i =q_i^{\prime}, 
\quad i=1,\cdots,n.
$$
Since $p_i=\left( \begin{array}{ll} a_i & b_i\end{array}\right)$, $p_i^\prime=\left(\begin{array}{ll}a_i^\prime & b_i^\prime\end{array}\right)$ with $a_i=a_i^\prime=0$ for $i \in S$ and $b_i=b_i^\prime=0$ for $i \in S^c$, while 
$q_i=\left( \begin{array}{ll} c_i & d_i \end{array}\right)^t$, $q_i^\prime=\left( \begin{array}{ll} c_i^\prime & d_i^\prime \end{array}\right)^t$
with $d_i\,=\,d_i^\prime\,=\,0$ for $i \in S$ and $c_i\,=\,c_i^\prime=0$ for $i 
\,\in\, S^c$, we conclude that
$$
A= \left(\begin{array}{cc}\alpha & 0 \\ 0 & \overline{\alpha} \end{array}\right)
$$
with $\alpha=e^{\sqrt{-1}\,\theta_0} \in S^1$. Then we have
$$
\left(\begin{array}{l} \operatorname{Re}{(b^\prime_i c^\prime_i)} \\ \\ 
\operatorname{Im}{(b^\prime_i c^\prime_i)} \\ \\ \frac{\lvert b^\prime_i\rvert^2+\lvert c^\prime_i \rvert^2}{2} \end{array}\right) = A_{-2\theta_0} \left(\begin{array}{l}\operatorname{Re}{(b_i c_i)} \\ \\ \operatorname{Im}{(b_i c_i)} \\ \\ \frac{\lvert b_i\rvert^2+\lvert c_i \rvert^2}{2} \end{array}\right)
\,\,\text{for $i \in S$,}
$$
and
$$
\left(\begin{array}{r} \operatorname{Re}{(a^\prime_i d^\prime_i)} \\ \\ - \operatorname{Im}{(a^\prime_i d^\prime_i)} \\ \\ - \frac{\lvert a^\prime_i \rvert^2+ \lvert d^\prime_i \rvert^2}{2} \end{array}\right) = A_{-2\theta_o} \left( 
\begin{array}{r} \operatorname{Re}{(a_i d_i)} \\ \\ - \operatorname{Im}{(a_i d_i)} \\ \\ - \frac{\lvert a_i\rvert^2+\lvert d_i \rvert^2}{2} \end{array} \right)\,\, \text{for $i \in S^c$},
$$
where $A_{-2\theta_0}$ is a rotation in ${\rm SU}(1,1)$. Therefore, it follows
that $\varphi$ is well-defined.

To show that $\varphi$ is injective, let us consider two points $[p,q]_\R, 
[p^\prime,q^\prime]_\R \,\in\, Z_S$ with
$\varphi([p,q]_\R)\,=\,\varphi([p^\prime,q^\prime]_\R) $. Then, writing 
$$ p_i=\left( \begin{array}{ll}0 & b_i\end{array}\right), \,\, p_i^\prime=\left( \begin{array}{ll}0 & b_i^\prime \end{array}\right)\quad \text{and} \quad q_i=\left( \begin{array}{l} c_i \\ 0 \end{array}\right),\,\, q_i^\prime=\left( \begin{array}{l} c_i^\prime \\ 0 \end{array}\right), \,\, \text{for $i\in S$} 
$$
with $\sum_{i\in S} b_i c_i=\sum_{i\in S} b^\prime_i c^\prime_i=0$ (cf. \eqref{eq:zero}),
and 
$$ p_i=\left( \begin{array}{ll} a_i & 0 \end{array}\right), \,\, p_i^\prime=\left( \begin{array}{ll} a_i^\prime & 0 \end{array}\right)\quad \text{and} \quad q_i=\left( \begin{array}{l} 0 \\ d_i \end{array}\right),\,\, q_i^\prime=\left( \begin{array}{l} 0 \\ d_i^\prime \end{array}\right), \,\, \text{for $i\in S^c$}, 
$$
with $\sum_{i\in S^c} a_i d_i=\sum_{i\in S^c} a^\prime_i d^\prime_i=0$, there exists an Euclidean rotation $A_{\theta_0}$ by an angle $\theta_0$ on the $(x,y)$-plane
such that
\begin{equation}\label{eq:prime1}
\left(\begin{array}{l} \operatorname{Re}{(b^\prime_i c^\prime_i)} \\ \\ \operatorname{Im}{(b^\prime_i c^\prime_i)} \\ \\ \frac{\lvert b^\prime_i\rvert^2+\lvert c^\prime_i \rvert^2}{2} \end{array}\right) = A_{\theta_0} \left(\begin{array}{l}\operatorname{Re}{(b_i c_i)} \\ \\ \operatorname{Im}{(b_i c_i)} \\ \\ \frac{\lvert b_i\rvert^2+\lvert c_i \rvert^2}{2} \end{array}\right)
\,\,\text{for $i \in S$,}
\end{equation}
and
\begin{equation}\label{eq:prime2}
\left(\begin{array}{r} \operatorname{Re}{(a^\prime_i d^\prime_i)} \\ \\ - \operatorname{Im}{(a^\prime_i d^\prime_i)} \\ \\ - \frac{\lvert a^\prime_i \rvert^2+ \lvert d^\prime_i \rvert^2}{2} \end{array}\right) = A_{\theta_0} \left( 
\begin{array}{r} \operatorname{Re}{(a_i d_i)} \\ \\ - \operatorname{Im}{(a_i d_i)} \\ \\ - \frac{\lvert a_i\rvert^2+\lvert d_i \rvert^2}{2} \end{array} \right)\,\, \text{for $i \in S^c$}.
\end{equation}
Indeed, if the two vectors on the left-hand side of \eqref{eq:prime1} and 
\eqref{eq:prime2} were not obtained from the corresponding vectors on the 
right-hand side by an Euclidean rotation, but by an element of ${\rm SU}(1,1)$ 
involving a boost, they would fail to satisfy the condition
$$
\sum_{i\in S} b_i^\prime c_i^\prime \,=\, \sum_{i\in S^c} a_i^\prime d_i^\prime 
\,=\,0\, . 
$$
We conclude that 
$$
b_i^\prime c_i^\prime \,= \,e^{\sqrt{-1}\, \theta_0} b_i c_i, \quad \text{and} \quad 
\lvert b_i^\prime \rvert^2 + \lvert c_i^\prime \rvert^2 = \lvert b_i \rvert^2 + \lvert c_i\rvert^2, \quad \text{for $i\in S$},
$$
while
$$
a_i^\prime d_i^\prime \,=\, e^{-\sqrt{-1}\, \theta_0} a_i d_i, \quad \text{and} \quad 
\lvert a_i^\prime \rvert^2 + \lvert d_i^\prime \rvert^2 = \lvert a_i \rvert^2 + \lvert d_i\rvert^2, \quad \text{for $i\in S^c$},
$$
and so 
$$
p_i^\prime \,= \,p_i A \quad \text{and} \quad q_i^\prime = A^{-1}q_i, \quad 
i\,=\,1,\cdots,n
$$
with $A\,= \left( \begin{array}{cc} e^{-\sqrt{-1}\, \theta_0/2 } & 0 \\ 0 & e^{\sqrt{-1}\, 
\theta_0/2 } \end{array}\right)$, implying that $[p,q]_\R=[p^\prime,q^\prime]_\R$.

Let us now see that $\varphi$ is surjective. For that, take any element 
$[v]\in M^{\lvert S\rvert,\lvert S^c \rvert}(\alpha)$. Using the ${\rm 
SU}(1,1)$--action, the $(k_1+1)$-th vertex can be placed
on the $t$-axis (so that $\sum_{i=1}^{k_1} v_i$ is a vector along the $t$-axis). Therefore,
 we may assume that $[v]$ is represented by a polygon with the first $\lvert S\rvert$ sides being
$(x_i,y_i,t_i)$ with $t_i>0$, on the positive time-like cone and the
last $n-\lvert S\rvert$ sides being $(x_i,y_i,-t_i)$ with $t_i>0$, in the past, satisfying the additional conditions
$$
\sum_{i=1}^{k_1} x_i=\sum_{i=1}^{k_1} y_i =\sum_{i=k_1+1}^n x_i = \sum_{i=k_1+1}^n y_i=0.
$$
Then $[v]$ is the image of the hyperpolygon $[p,q]_\R$, where
$$
p_i= \left( \begin{array}{ll}0 & \frac{1}{l_i}(x_i+\sqrt{-1}\, y_i)\end{array}\right), \quad q_i= \left(\begin{array}{l} l_i \\ 0 \end{array}\right) \,\, \text{for $i\in S$},
$$
and 
$$ 
p_i= \left( \begin{array}{ll} \frac{1}{l_i} (x_i-\sqrt{-1} y_i) & 0 \end{array}\right), \quad q_i=\left( \begin{array}{c} 0 \\ l_i \end{array}\right)\,\, \text{for $i\in S^c$},
$$
with 
$$
l_i \,=\, \sqrt{\alpha_i +\sqrt{\alpha^2+\lvert x_i +\sqrt{-1}\, y_i \rvert^2}} 
\,=\, \sqrt{\alpha_i+t},\quad i\,=\,1,\cdots, n\, .
$$
Here $[p,q]_\R\in Z_S$ since 
$$
\sum_{i\in S} b_i c_i \,=\,\sum_{i=1}^{k_1} (x_i + \sqrt{-1}\, y_i )\,= \,0,
\quad  \sum_{i \in S^c} a_i d_i= \sum_{i=k_1+1}^n ( x_i - \sqrt{-1}\, y_i )\, ,
$$
and
$$
\lvert c_i \rvert^2 - \lvert b_i\rvert^2 =2 \alpha_i \,\,\text{for all $i \in S$} \quad \text{while} \quad \lvert d_i\rvert^2 - \lvert a_i\rvert^2 =2 \alpha_i \,\,\text{for all $i \in S^c$} ,
$$
where as usual we write $p_i= \left( \begin{array}{ll} a_i& b_i 
\end{array}\right)$ and $q_i= \left( \begin{array}{ll} c_i& d_i \end{array}\right)^t$, for $i=1,\cdots,n$. 

Note that clearly $\varphi$ and its inverse are differentiable and the theorem 
follows.
\end{proof}

\begin{rem}
{\rm Note that when $\lvert S\rvert=n-1$ we obtain that the space $M^{n-1,1}(\alpha)$, 
which we already knew is compact, is, in fact, diffeomorphic to $\C\P^{n-3}$.}
\end{rem}

Theorem~\ref{thm:polMink} allows us draw several conclusions on the polygon 
spaces in Minkowski $3$-space which are immediate consequences of Theorem 
\ref{thm:defret} and Theorem \ref{thm:Poinc}.

\begin{theorem}
Let $M^{k_1,k_2}(\alpha)$ be the moduli space of closed polygons in Minkowski 
$3$-space that have the first $k_1$ sides in the future time-like cone and the 
last $k_2$ in the past, such that the Minkowski length of the $i$-th side is 
$\alpha_i$. Assume without loss of generality that $\sum_{i=1}^{k_1} \alpha_i\, < 
\,\sum_{i=k_1+1}^n \alpha_i$. Then,
\begin{enumerate}
\item[(i)] $M^{k_1,k_2}(\alpha)$ admits a deformation retraction to $\C \P^{k_1-2}$, and

\item[(ii)] the Poincar\'{e} polynomial of $M^{k_1,k_2}(\alpha)$ is
$$
P_t(M^{k_1,k_2}(\alpha))\,=\,P_t(\C \P^{k_1-2})=1+t+\cdots+t^{2(\lvert S\rvert 
-2)}\, .
$$
\end{enumerate}
\end{theorem}

\section{An Example}\label{sec:ex}

As an example, we consider the case where $n=4$. Let $\mathcal{H}(\beta)$ be 
the moduli space of parabolic Higgs bundles $(E,\Phi)$ of rank two over $\C \P^1$ with 
four parabolic points, where the underlying holomorphic vector bundle is trivial 
and $\beta$ is generic. Let
$$
\alpha_i\,:=\,\beta_2(x_i)-\beta_1(x_i), \quad i=1,\cdots,4\, .
$$
Since a subset of $\{1,2,3,4\}$ is either short or long, we know that there are 
exactly three short sets of cardinality two for any value of 
$\alpha=(\alpha_1,\cdots,\alpha_4)$. Let us denote these sets by $S_1$, $S_2$ and 
$S_3$. Then the fixed point set of the involution in \eqref{eq:invPHB} has 
exactly $4$ connected components
$$
\mathcal{M}_{\beta,2,0}, \mathcal{Z}_{S_1}, \mathcal{Z}_{S_2}, \mathcal{Z}_{S_3} \quad \text{or} \quad \mathcal{Z}_{S_1}, \mathcal{Z}_{S_2}, \mathcal{Z}_{S_3}, \mathcal{Z}_{\widetilde{S}},
$$
according to whether $\mathcal{M}_{\beta,2,0}$ is empty or nonempty, where 
$\widetilde{S}$ is a short set of cardinality $3$ which we know exists exactly 
when $\mathcal{M}_{\beta,2,0}\,=\,\emptyset$ \cite{GM,BY}. 

If $\mathcal{M}_{\beta,2,0}\,\neq\, \emptyset$, then this space 
$\mathcal{M}_{\beta,2,0}$ is a compact 
toric 
manifold of dimension two, therefore diffeomorphic to $\C \P^1$. Indeed, let us 
assume without loss of generality that $\alpha_1\neq \alpha_2$ (note that 
$\alpha$ is generic) and consider the diagonal $d_2:=u_1+u_2$ connecting the 
origin to the third vertex of the polygon. For each intermediate value of the 
length of $d_2$, we have a circle of possible classes of polygons obtained by 
rotating the first two sides of the polygon around the diagonal, while fixing 
the other two. The minimum and maximum values of this length are
$$
\max \left\{\lvert \alpha_1-\alpha_2\rvert\, , \lvert \alpha_3- \alpha_4\rvert 
\right\} \quad \text{and} \quad \min\left\{ \alpha_1+\alpha_2\, , 
\alpha_3+\alpha_4 \right\}
$$
respectively, in which cases we only have one possible polygon. Note that this 
length is the 
moment map for the bending flow obtained by rotating the first two sides of the polygon around the diagonal.

If $\mathcal{M}_{\beta,2,0}\,=\,\emptyset$ then, since $\lvert \widetilde{S} 
\rvert =3$, we have that 
$\mathcal{Z}_{\widetilde{S}}=\mathcal{M}_{\widetilde{S}}$ is a connected 
component of $\mathcal{H}(\beta)^{\Z/2\Z}$ diffeomorphic to $\C \P^1$.

Let us now consider $\mathcal{Z}_{S_i}$. By Theorem~\ref{thm:polMink} we know 
that this space is diffeomorphic to $M^{2,2}(\alpha)$ formed by classes of 
closed polygons in Minkowski $3$-space with the first two sides $u_1,u_2$ in 
the future time-like cone and the last two, namely $u_3$ and $u_4$, in the past, where each 
side $u_i$ has Minkowski length $\alpha_i$. Let us again consider the diagonal
$d_2=u_1+u_2$ connecting the origin to the third vertex of the polygon. This 
vector is also a future time-like vector and we can consider its Minkowski 
length $\ell$. Note that if we place the first vertex at the origin and use the 
${\rm SU}(1,1)$--action to place the third vertex on the $t$-axis, then the 
bending flow can be described as a rotation of the vectors $u_1,u_2$ around the 
$t$-axis with a constant angular speed while fixing the other two vectors. Hence, 
$\mathcal{Z}_{S_i}$ is a non-compact toric manifold with moment map $\ell$. By 
the reversed triangle inequality we have that $\ell$ has the minimum value
$$\max \left\{ \alpha_1+\alpha_2\, ,\alpha_3+\alpha_4\right\}$$ 
which is attained at just one point (the polygon with two sides aligned along the $t$-axis) and has no other critical value. We conclude that $\mathcal{Z}_{S_i}$ is diffeomorphic to $\C$. 

In all cases we conclude that $\mathcal{H}(\beta)^{\Z/2\Z}$ has one compact 
connected component diffeomorphic to $\C \P^1$ and three non-compact components diffeomorphic to $\C$.


\end{document}